\documentclass[10pt]{amsart}

\usepackage{amsmath}
\usepackage{amssymb}
\usepackage{bm}
\usepackage{graphicx}
\usepackage{psfrag}
\usepackage{color}
\usepackage{hyperref}
\hypersetup{colorlinks=true, linkcolor=blue, citecolor=magenta, urlcolor=wine}
\usepackage{url}
\usepackage{algpseudocode}
\usepackage{fancyhdr}
\usepackage{mathtools}
\usepackage{tikz-cd}
\usepackage{xy}
\input xy
\xyoption{all}
\usepackage{stmaryrd}
\usepackage{calrsfs}

\newtheorem*{maintheorem*}{Main Theorem}
\newtheorem{theorem}{Theorem}[section]
\newtheorem{prop}[theorem]{Proposition}

\newtheorem{lem}[theorem]{Lemma}
\newtheorem{cor}[theorem]{Corollary}

\theoremstyle{definition}
\newtheorem{defn}[theorem]{Definition}

\newtheorem{ex}[theorem]{Example}

\newtheorem{question}[theorem]{Question}
\numberwithin{equation}{section}

\newcommand{\cc}{\mathbb{C}}
\newcommand{\ff}{\mathbb{F}}
\newcommand{\nn}{\mathbb{N}}
\newcommand{\pp}{\mathbb{P}}
\newcommand{\qq}{\mathbb{Q}}
\newcommand{\rr}{\mathbb{R}}
\newcommand{\zz}{\mathbb{Z}}

\providecommand\ldb{\llbracket}
\providecommand\rdb{\rrbracket}

\newcommand{\gp}{\text{gp}}
\newcommand{\qf}{\text{qf}}

\newcommand{\uu}{\mathcal{U}}

\newcommand{\norm}[1]{\left\lVert#1\right\rVert}



\begin{document}
	\mbox{}
	\title{Hereditary atomicity and ACCP in abelian groups} 
	
	\author{Felix Gotti}
	\address{Department of Mathematics\\MIT\\Cambridge, MA 02139}
	\email{fgotti@mit.edu}

\date{\today}

\begin{abstract}
	A cancellative and commutative monoid $M$ is atomic if every non-invertible element of~$M$ factors into irreducibles (also called atoms), and $M$ is hereditarily atomic if every submonoid of $M$ is atomic. In addition, $M$ is hereditary ACCP if every submonoid of $M$ satisfies the ascending chain condition on principal ideals (ACCP). Our primary purpose in this paper is to determine which abelian groups are hereditarily atomic. In doing so, we discover that in the class of abelian groups the properties of being hereditarily atomic and being hereditary ACCP are equivalent. Once we have determined the abelian groups that are hereditarily atomic, we will use this knowledge to determine the commutative group algebras that are hereditarily atomic, that is, the commutative group algebras satisfying that all their subrings are atomic. The interplay between atomicity and the ACCP is a subject of current active investigation. Throughout our journey, we will discuss several examples connecting (hereditary) atomicity and the ACCP, including, for each integer $d$ with $d \ge 2$, a construction of a rank-$d$ additive submonoid of $\mathbb{Z}^d$ that is atomic but does not satisfy the ACCP.
\end{abstract}

\dedicatory{Dedicado a mi padre, Juan Gotti, \\ por todo lo que me enseñó y me dejó}

\bigskip
\maketitle

\bigskip
\section{Introduction}
\label{sec:intro}

Following P. Cohn~\cite{pC68}, we say that a cancellative and commutative monoid is atomic provided that every non-invertible element factors into irreducibles (also called atoms), while an integral domain is atomic provided that its multiplicative monoid is atomic. Although many relevant classes of integral domains, including Noetherian domains and Krull domains, consist of atomic domains, until the nineties the notion of atomicity had been studied only in connection to other algebraic properties: for instance, A. Grams~\cite{aG74} and A. Zaks~\cite{aZ82} investigated atomicity in the context of integral domains and in connection with the ascending chain condition on principal ideals (ACCP). However, in the celebrated paper~\cite{AAZ90} by D. D. Anderson, D. F. Anderson, and M.~Zafrullah, published in 1990, the authors proposed a diagram of nested classes of atomic domains as a methodology to study the phenomenon of multiple factorizations in the context of integral domains. This paper, along with the paper~\cite{fHK92} by F. Halter-Koch (where two of the atomic notions introduced in~\cite{AAZ90} were generalized to cancellative monoids and investigated in this more general setting), motivated a significant interest for the notion of atomicity. Since then, atomicity has played a central role in the factorization and ideal theory of cancellative and commutative monoids, which have been not only systematically studied since then but also extended to non-cancellative~\cite{FGKT17} and non-commutative settings~\cite{BS15,aG16}.
\smallskip

As we mentioned before, since the pioneer papers on atomic domains, the phenomenon of atomicity has been often investigated in connection with the ACCP. A cancellative and commutative monoid (or an integral domain) is said to satisfy the ACCP provided that every ascending chain of principal ideals eventually stabilizes. Although one can readily show that every cancellative and commutative monoid (and so every integral domain) satisfying the ACCP is atomic, not every atomic domain satisfies the ACCP. However, none of the known counterexamples are easy to construct. The first of such constructions dates back to 1974 and it was carried out by Grams in~\cite[Section~1]{aG74} to correct the (wrong) assertion stated in~\cite[Proposition~1]{pC68} on the equivalence of atomicity and the ACCP in the class of integral domains. Further examples of atomic domains that do not satisfy the ACCP have been constructed by Zaks~\cite{aZ82} in the eighties and by M. Roitman~\cite{mR93} in the nineties. More recently, J.~Boynton and J. Coykendall~\cite{BC19} constructed a new example using certain pullbacks of commutative rings. Even more recently, various methods to construct further examples were given by B. Li and the author~\cite{GL23}, where they introduced the weak ACCP, an atomic notion that falls strictly between the property of being atomic and that of satisfying the ACCP.
\smallskip

Although atomicity has been systematically studied for more than three decades, the notion of hereditary atomicity has been investigated only recently, by J. Coykendall, R. Hasenauer, and the author~\cite{CGH21}. We say that a cancellative and commutative monoid $M$ (resp., an integral domain $R$) is hereditarily atomic if every submonoid of $M$ (resp.,  subring of $R$) is atomic. In the same vein, we can define the notion of being hereditary ACCP. As one of the primary results in~\cite{CGH21}, the authors characterized the fields that are hereditarily atomic domains. In the last section of the same paper, the authors presented some partial results towards a comparison between being hereditarily atomic and being hereditary ACCP (in the setting of integral domains). In the same section of~\cite{CGH21}, the authors were not able to decide whether being hereditarily atomic and being hereditary ACCP are equivalent properties in the class consisting of all integral domains. It is worth noticing that the same question is still open if one replaces the class of integral domains by that of cancellative and commutative monoids. However, a similar question has been recently answered positively for the class of reduced torsion-free monoids~\cite[Theorem 3.1]{GV23}.
\smallskip

Motivated by the characterization provided in~\cite{CGH21} of the fields that are hereditarily atomic domains, in this paper we will determine the abelian groups that are hereditarily atomic monoids. In addition, we will see that, in the class of abelian groups, the property of being hereditarily atomic and that of being hereditary ACCP are equivalent. This is parallel to the statement conjectured in~\cite{CGH21} about the equivalence of the same two hereditary properties in the context of integral domains. For each integer~$d$ with $d \ge 2$, it turns out that the free abelian group $\zz^d$ is not hereditarily atomic: we will explicitly construct a rank-$d$ atomic additive submonoid of $\zz^d$ that is not hereditarily atomic. Since these monoids are reduced by construction, it follows from \cite[Theorem~3.1]{GV23} that they are (the first known) submonoids of a free abelian group that are atomic but do not satisfy the ACCP. As for the case of abelian groups, we determine the (commutative) group algebras that are hereditarily atomic domains. Along the way, we give some examples and classes of hereditarily atomic monoid algebras, atomic monoid algebras that do not satisfy the ACCP, and hereditary ACCP monoid algebras that are not bounded factorization domains.
\smallskip

Here is a brief roadmap of the results and the structure of this paper. In Section~\ref{sec:background}, we provide most of the notation and terminology as well as some of the results we shall be using in latter sections. In Section~\ref{sec:abelian groups}, we investigate the properties of being hereditarily atomic and being hereditary ACCP in the class of abelian groups. We prove in Theorem~\ref{thm:HA groups} that these two properties are equivalent in the whole class of abelian groups, and we actually determine the abelian groups satisfying the two properties. Then for the most elementary abelian groups that do not satisfy these properties, namely, $\qq$ and $\zz^d$ (for all $d \ge 2$), we exhibit atomic submonoids of such groups that do not satisfy the ACCP (this is done in Example~\ref{ex:atomic PM without ACCP} for $\qq$ and in Proposition~\ref{prop:atomic monoid that is not ACCP} for $\zz^d$). In Section~\ref{sec:commutative group rings}, we investigate the properties of being hereditarily atomic and being hereditary ACCP in the class of commutative group algebras. Using the main result established in Section~\ref{sec:abelian groups}, in Proposition~\ref{prop:group rings} we determine the commutative group algebras that are hereditarily atomic. In contrast to Proposition~\ref{prop:group rings}, we exhibit in Proposition~\ref{prop:antimatter group algebras} three classes of antimatter group algebras, complementing the main constructions of antimatter monoid algebras provided by Anderson et al. in~\cite{ACHZ07}. For the sake of parallelism with Section~\ref{sec:abelian groups} and also to highlight some attractive related questions (Question~\ref{quest:new atomic monoid algebras without ACCP}), in Example~\ref{ex:Gotti-Li's example} we describe a subring of the non-hereditarily atomic commutative group algebra $F[\rr]$ (for any prescribed field $F$) that is atomic but does not satisfy the ACCP (this construction is taken from the recent paper~\cite{GL22} and, therefore, we only offer a brief description, avoiding the most involved technical aspects). In Subsection~\ref{subsec:HACCP and the BFP}, we conclude with a few words about the property of being hereditary ACCP and that of being a bounded factorization domain in the context of commutative group algebras.

\bigskip
\section{Background}
\label{sec:background}

\smallskip
\subsection{General Notation} 

Following standard notation, we let $\zz$, $\qq$, $\rr$, and $\cc$ denote the set integers, rational numbers, real numbers, and complex numbers, respectively. In addition, we let $\nn$, $\nn_0$, and~$\pp$ denote the set of positive integers, nonnegative integers, and primes, respectively. For $p \in \pp$ and $n \in \nn$, we let $\ff_{p^n}$ be the finite field of cardinality $p^n$. For $b,c \in \zz$ with $b \le c$, we let $\ldb b,c \rdb$ denote the set of integers between~$b$ and~$c$, i.e., $\ldb b,c \rdb = \{n \in \zz \mid b \le n \le c\}$. Also, for $S \subseteq \rr$ and $r \in \rr$, we set $S_{\ge r} = \{s \in S \mid s \ge r\}$ and $S_{> r} = \{s \in S \mid s > r\}$.

\smallskip
\subsection{Abelian Groups and Commutative Monoids} 

Let $G$ be an (additive) abelian group. There is a natural group homomorphism $\varphi \colon G \to V$, where $V$ is the $\qq$-vector space $\qq \otimes_\zz G$ and $\varphi \colon g \mapsto 1 \otimes g$. It is clear that the rank of $G$ as $\zz$-module, here denoted by $\text{rank} \, G$, coincides with the dimension of the vector space~$V$. When $G$ is torsion-free, the flatness of~$\qq$ as a $\zz$-module makes $\varphi$ an embedding, and so we can identify $G$ with a subgroup of $V$ via $\varphi$. We say that a subset $\{g_1, \dots, g_k\}$ is an \emph{integrally independent} subset of~$G$ if the only integral combination of $g_1, \dots, g_k$ equaling $0$ in $G$ is the trivial combination (all coefficients taken to be zero). Observe that a finite subset of~$G$ is integrally independent if and only if its image under $\varphi$ is linearly independent in the vector space~$V$. For a total order $\preceq$ on $G$, the pair $(G, \preceq)$ is called a \emph{totally ordered group} provided that $\preceq$ is compatible with the operation of $G$, that is, for any $g, h, k \in G$ the relation $g \preceq h$ guarantees that $g+k \preceq h+k$. If $(G, \preceq)$ is a totally ordered group, then the additive subset $\{g \in G \mid 0 \preceq g\}$ (resp., $\{g \in G \mid 0 \prec g\}$) is called the \emph{nonnegative} (resp., \emph{positive}) \emph{cone} of $(G,\preceq)$. A totally ordered group $(G, \preceq)$ is called an \emph{Archimedean group} provided that for all $g$ and $h$ in the positive cone of $(G, \preceq)$ there exists $n \in \nn$ such that $h \preceq ng$. A theorem by H\"older states that every Archimedean group is isomorphic (as an ordered group) to an additive subgroup of $\rr$.
\smallskip

A semigroup with an identity element is called a \emph{monoid}. Throughout this paper, we will tacitly assume that every monoid we deal with is cancellative and commutative. For the rest of this section, let $M$ be a monoid written additively. We let $M^\bullet$ denote the set of nonzero elements of $M$. The group consisting of all the invertible elements of $M$ is denoted by $\uu(M)$, and $M$ is called \emph{reduced} when $\uu(M)$ is trivial. As for abelian groups, a pair $(M, \preceq)$ is a \emph{totally ordered monoid} if $\preceq$ is a total order on $M$ such that, for all $b,c,d \in M$, the relation $b \preceq c$ implies that $b + d \preceq c + d$. Observe that the nonnegative cone of any totally ordered abelian group is a reduced and totally ordered monoid. The \emph{difference group} of $M$, here denoted by $\gp(M)$, is the unique abelian group (up to isomorphism) satisfying that any abelian group containing a homomorphic image of~$M$ also contains a homomorphic image of $\gp(M)$. The monoid $M$ is called \emph{torsion-free} provided that $\gp(M)$ is a torsion-free abelian group. The \emph{rank} of $M$, denoted by $\text{rank} \, M$, is by definition the rank of $\gp(M)$. A \emph{submonoid} of~$M$ is a subset closed under addition that contains the identity element of $M$. If $S$ is a subset of $M$, then we let $\langle S \rangle$ denote the submonoid of $M$ generated by $S$. If $M = \langle S \rangle$ for some finite set $S$, then $M$ is said to be \emph{finitely generated}. Additive submonoids of $\nn_0$ are always finitely generated; they are well-studied monoids, often referred to as \emph{numerical monoids}.
\smallskip

A non-invertible element $a \in M$ is called an \emph{atom} provided that for all $b,c \in M$ the equality $a = b+c$ implies that either $b \in \uu(M)$ or $c \in \uu(M)$. The set consisting of all the atoms of $M$ is denoted by $\mathcal{A}(M)$. Following Coykendall, Dobbs, and Mullin~\cite{CDM99}, we say that $M$ is \emph{antimatter} if $\mathcal{A}(M)$ is empty. On the other hand, following Cohn~\cite{pC68}, we say that $M$ is \emph{atomic} if every non-invertible element of $M$ can be written as a sum of atoms. Observe that the only monoids that are simultaneously antimatter and atomic are the abelian groups. Hereditary atomicity is one of the notions we are primarily concerned with in this paper.

\begin{defn} \label{def:hereditarily atomic monoid}
A monoid $M$ is \emph{hereditarily atomic} if every submonoid of $M$ is atomic.
\end{defn} 

A subset $I$ of $M$ is an \emph{ideal} of $M$ if $I + M \subseteq I$ or, equivalently, if $I + M = I$. An ideal $I$ of $M$ is \emph{principal} if $I = b + M$ for some $b \in M$. If $c \in b + M$, then we say that~$b$ \emph{(additively) divides} $c$ in~$M$, in which case we write $b \mid_M c$. A submonoid $N$ of $M$ is called \emph{divisor-closed} if for any $b \in M$ and $c \in N$, the divisibility relation $b \mid_M c$ implies that $b \in N$. The monoid $M$ satisfies the \emph{ascending chain condition on principal ideals} (ACCP) if every ascending chain of principal ideals of $M$ eventually stabilizes. If a monoid satisfies the ACCP, then it is atomic \cite[Proposition~1.1.4]{GH06}. An atomic monoid may not satisfy the ACCP, as we shall see in Example~\ref{ex:atomic PM without ACCP} and Proposition~\ref{prop:atomic monoid that is not ACCP}. Since every abelian group has exactly one principal ideal, namely the whole group, every abelian group trivially satisfies the ACCP and is, therefore, atomic. In addition, finitely generated monoids satisfy the ACCP (see, for instance, \cite[Proposition~2.7.8]{GH06}), and so they are atomic. In the direction of Definition~\ref{def:hereditarily atomic monoid}, we say that $M$ is \emph{hereditary ACCP} provided that every submonoid of $M$ satisfies the ACCP.

\smallskip
\subsection{Integral Domains}

Let $R$ be an integral domain. We let $R^\ast$ and $R^\times$ denote the multiplicative monoid and the group of units of $R$, respectively. In addition, we let $\qf(R)$ denote the quotient field of~$R$. An \emph{overring} of~$R$ is an intermediate ring of the ring extension $R \subseteq \qf(R)$. The integral domain~$R$ is \emph{atomic} (resp., satisfies the \emph{ascending chain condition on principal ideals} (ACCP)) provided that its multiplicative monoid $R^*$ is atomic (resp., satisfies the ACCP). We let $\mathcal{A}(R)$ denote the set of all irreducible elements of $R$: the reuse of notation is convenient and justified by the fact that the irreducibles of $R$ are precisely the atoms of the multiplicative monoid $R^*$. The following notion, introduced in~\cite{CGH21}, plays a central role in Section~\ref{sec:commutative group rings}. 

\begin{defn}
	An integral domain $R$ is \emph{hereditarily atomic} if every subring of $R$ is atomic (subrings of $R$ are assumed to contain the identity of $R$).
\end{defn}

\noindent Observe that if the multiplicative monoid $R^*$ of $R$ is hereditarily atomic, then $R$ is a hereditarily atomic integral domain. However, it is worth emphasizing that the converse of this observation does not hold in general.

\begin{ex}
	Since $\zz$ is a Dedekind domain, each overring of $\zz$ is also a Dedekind domain and, therefore, a Noetherian domain. This, along with the fact that every subring of $\qq$ contains $\zz$, allows us to conclude that $\qq$ is a hereditarily atomic integral domain (this is a special case of either \cite[Theorem~4.4]{CGH21} or \cite[Theorem]{rG70}). However, we claim that the submonoid $M := \qq_{\ge 1}$ of the multiplicative monoid of $\qq$ is not atomic: indeed, $M$ is an antimatter monoid. Checking this last assertion amounts to observing that for any $q \in \qq_{> 1}$ we can pick $n \in \nn$ sufficiently large so that $\frac{n}{n+1}q > 1$ and then decompose $q$ in $M$ as $q = \big( \frac{n}{n+1}q \big) \big( \frac{n+1}{n} \big)$. As a consequence, the multiplicative monoid of $\qq$ is not hereditarily atomic.
\end{ex}

For an integral domain $R$ and a monoid $M$, we let $R[x;M]$ denote the ring consisting of all polynomial expressions in an indeterminate $x$ with coefficients in~$R$ and exponents in~$M$ (with addition and multiplication defined as for standard polynomials). Following Gilmer~\cite{rG84}, we will write $R[M]$ instead of $R[x;M]$ provided that we see no risk of ambiguity. As $R[M]$ is an algebra over $R$, the former is often called the \emph{monoid algebra} of $M$ over $R$ or, simply, a \emph{monoid algebra}. From now on, we adopt this terminology. In the special case where $M$ is a group, $R[M]$ is called a \emph{group algebra} (\emph{over}~$R$). Throughout this paper, monoids of exponents of monoid algebras are additively written. When the monoid~$M$ is torsion-free, it follows from \cite[Theorem~8.1]{rG84} that $R[M]$ is an integral domain and, moreover,
\[
	R[M]^\times = \{ux^m \mid u \in R^\times \ \text{and} \ m \in \uu(M)\}
\]
by \cite[Theorem~11.1]{rG84}. If $(M, \preceq)$ is a totally ordered monoid, then one can write each nonzero element $f \in R[M]^*$ uniquely in the following canonical form: $f = \alpha_n x^{d_n} + \dots + \alpha_1 x^{d_1}$ for some coefficients $\alpha_1, \dots, \alpha_n \in R^*$ and exponents $d_1, \dots, d_n \in M$ satisfying that $d_1 \prec \cdots \prec d_n$. In this case, the elements $\deg \, f := d_n$ and $\text{ord} \, f := d_0$ are called the \emph{degree} and the \emph{order} of $f$, respectively. A fair exposition of some of the most relevant advances of monoid algebras until 1984 was given by Gilmer in~\cite{rG84}.

\smallskip
\subsection{Euclidean Geometry and Convexity}

Throughout this subsection, we fix $d \in \nn$. For any $v \in \rr^d$, we let $\norm{v}$ denote the Euclidean norm of~$v$. For $v \in \rr^d$ and a subspace $W$ of $\rr^d$, we let $d(v,W)$ denote the Euclidean distance from $v$ to $W$; that is, $d(v,W) := \min \{ \|v-w\|  \mid w \in W \}$. We always consider the space $\rr^d$ endowed with the topology induced by the Euclidean norm. We denote the standard inner product of $\rr^d$ by $\langle \, , \rangle$; that is, $\langle v,w \rangle = \sum_{i=1}^d v_i w_i$ for all $v = (v_1, \dots, v_d)$ and $w = (w_1, \dots, w_d)$ in $\rr^d$. A nonempty subset $C$ of $\rr^d$ is called a \emph{cone} provided that~$C$ is closed under nonnegative linear combinations. A cone~$C$ is called \emph{pointed} if $C \cap -C = \{0\}$. Let $S$ be a nonempty subset of $\rr^d$. The \emph{conic hull} of $S$, denoted by $\text{cone}_\rr(S)$, is defined as follows:
\[
	\text{cone}_\rr(S) := \bigg\{ \sum_{i=1}^n c_i v_i \ \bigg{|} \ n \in \nn, \ \text{and} \ v_i \in S \ \text{and} \ c_i \in \rr_{\ge 0} \ \text{for every} \ i \in \ldb 1,n \rdb \bigg\},
\]
which means that $\text{cone}_\rr(S)$ is the smallest cone in $\rr^d$ containing $S$. When we see no risk of ambiguity, we write $\text{cone}(S)$ instead of $\text{cone}_\rr(S)$.
\smallskip

Now assume that $d \ge 2$. Take $c_0, \dots, c_{d-1} \in \rr$, and let $H$ be an affine hyperplane of $\rr^d$ described by an equation $x_d =  c_0 + \sum_{i=1}^{d-1} c_i x_i$. The \emph{upper closed affine half-space} determined by $H$ is the set
\[
	H^+ := \bigg\{ (x_1, \dots, x_d) \in \rr^d \ \bigg{|} \  x_d \ge c_0 + \sum_{i=1}^{d-1} c_i x_i \bigg\}.
\]
We define the \emph{lower closed affine half-space} $H^-$ similarly but using the inequality $x_d \le c_0 + \sum_{i=1}^{d-1} c_i x_i$ instead. To easy notation, we call $H^+$ (resp., $H^-$) simply the upper half-space (resp., the lower half-space) determined by $H$.

\bigskip
\section{Abelian Groups}
\label{sec:abelian groups}

We begin considering the property of being hereditarily atomic in the class of abelian groups. In Subsection~\ref{subsec:HA abelian groups}, we determine the abelian groups that are hereditarily atomic. Next, in Subsection~\ref{subsec:inside abelian groups that are not HA}, we take a look at some simple examples of abelian groups that are not hereditarily atomic and, inside them, we identify certain submonoids that are interesting from the atomicity perspective: the submonoids in Proposition~\ref{prop:atomic monoid that is not ACCP} are produced using a novel construction, while the submonoids in parts~(a) and~(c) of Example~\ref{ex:atomic PM without ACCP} have been useful in commutative ring theory (see~\cite[Theorem~1.3]{aG74} and~\cite[Example~3.6]{GL22}, respectively).

\smallskip
\subsection{Hereditarily Atomic Abelian Groups}
\label{subsec:HA abelian groups}

The primary purpose of this section is to determine which abelian groups are hereditarily atomic, which we do in the next theorem. It is well known and not difficult to argue that every submonoid of the abelian group~$\zz$ is either a subgroup of~$\zz$ or a finitely generated monoid. Hence $\zz$ is a hereditarily atomic group. This is not the case for the free abelian group $\zz^2$, as the next example illustrates. For any $d \in \nn$, we refer to the elements of $\zz^d$ as \emph{lattice points} and the submonoids of the free abelian group $\zz^d$ as \emph{lattice monoids}.

\begin{ex} \label{ex:Z^2 is not HA}
	Consider the lattice monoid $M := (\nn \times \zz) \cup (\{0\} \times \nn_0)$ of the rank-$2$ free abelian group $\zz^2$. Observe that $M$ is reduced. Also, it is clear that $a := (0,1)$ is an atom of~$M$ and also that~$a$ divides each nonzero element of~$M$. This, along with the fact that $M$ is reduced, ensures that $\mathcal{A}(M) = \{a\}$. Because $\langle \mathcal{A}(M) \rangle = \{0\} \times \nn_0$, the monoid $M$ is not atomic. Hence the abelian group $\zz^2$ is not hereditarily atomic. Since, for each integer $d$ with $d \ge 3$, we can naturally embed $\zz^2$ as a subgroup of $\zz^d$, we infer that the abelian group $\zz^d$ is not hereditarily atomic.
	
	 We can also arrive to the same conclusion by observing first that $M$ is the nonnegative cone of the totally ordered group $(\zz^2, \preceq)$, where $\preceq$ denotes the lexicographical order (with priority on the first coordinate). Similarly, we can argue that, for each integer $d$ with $d \ge 3$, the nonnegative cone $M_d$ of the totally ordered group $(\zz^d, \preceq)$, where $\preceq$ is the lexicographical order, contains only one atom (namely, the minimum of $M_d^\bullet$), which divides any positive element in $M_d^\bullet$. Therefore $M_d^\bullet$ is a submonoid of $\zz^d$ that is not atomic. Thus, $\zz^d$ is not hereditarily atomic.
\end{ex}

With Example~\ref{ex:Z^2 is not HA} in mind, we proceed to establish the primary result of this section.

\begin{theorem} \label{thm:HA groups}
	Let $G$ be an abelian group, and let $T$ be the torsion subgroup of $G$. Then the following conditions are equivalent.
	\begin{enumerate}
		\item[(a)] $G$ is hereditary ACCP.
		\smallskip
		
		\item[(b)] $G$ is hereditarily atomic.
		\smallskip
		
		\item[(c)] $G/T$ is cyclic.
	\end{enumerate}
\end{theorem}

\begin{proof}
	(a) $\Rightarrow$ (b): This is clear as every monoid satisfying the ACCP is atomic.
	\smallskip
	
	(b) $\Rightarrow$ (c): Assume that $G$ is hereditarily atomic. We first prove that the quotient group $G/T$ has rank at most~$1$, and then we use this information to prove that $G/T$ is cyclic. 
	\smallskip
	
	Suppose, for the sake of a contradiction, that $\text{rank} \, G/T \ge 2$. Let $d$ be the rank of $G/T$, which is also the rank of $G$. Since $d \ge 2$, we can take $u,v \in G$ such that $\{u,v\}$ is an integrally independent subset of~$G$ (in particular, $u$ and $v$ are non-torsion elements of $G$). 
	Consider the totally ordered group $(\zz u + \zz v, \preceq)$, where $\preceq$ is the lexicographical order (with priority on the $u$-coordinate). Now set $M := (\nn u + \zz v) \, \bigcup \, \nn_0 v$, and observe that~$M$ is the nonnegative cone of the totally ordered group $(\zz u + \zz v, \preceq)$.
	Then $M$ is a reduced monoid. We claim that~$M$ is not atomic. Since $v$ is the minimum element of $M^\bullet$, which is the positive cone of $(\zz u + \zz v, \preceq)$, it follows that $v$ divides every nonzero element of $M$. Since $M$ is reduced, no element in $M \setminus \{0, v\}$ can divide $v$ in $M$. Hence $v$ is an atom of $M$. Since $v$ divides any nonzero element of $M$, the fact that $M$ is reduced then guarantees that $\mathcal{A}(M) = \{v\}$. 
	Because $M \neq \nn_0 v$, we conclude that $M$ is not atomic. As a result, $G$ is not hereditarily atomic, contradicting our initial assumption. Hence $\text{rank} \, G/T \le 1$.
	\smallskip
	
	We proceed to argue that $G/T$ is cyclic. Assume, towards a contradiction, that this is not the case. Then $G/T$ cannot be the trivial group and, therefore, $\text{rank} \, G/T = 1$. It is well known that every rank-$1$ torsion-free abelian group is isomorphic to an additive subgroup of $\qq$ (see, for instance, \cite[Section~24]{lF70}). It is also known that every additive subgroup of $\qq$ is either an infinite cyclic groups or a strictly ascending union of infinite cyclic groups. Since $G/T$ is not a cyclic group, we can write
	\[
		G/T = \bigcup_{n \in \nn} \zz (s_n + T),
	\]
	where $(s_n)_{n \ge 1}$ is a sequence with terms in $G$ such that $\zz (s_n + T) \subsetneq \zz (s_{n+1} + T)$ for every $n \in \nn$. Because $G/T$ is not the trivial group, after dropping the first term of $(s_n)_{n \ge 1}$, we can assume that $s_n \notin T$ for any $n \in \nn$. As $\zz(s_n + T) = \zz(-s_n + T)$ for every $n \in \nn$, after replacing $s_n$ by $-s_n$ when necessary, we can assume that $s_n + T \in \nn(s_{n+1} + T)$ for every $n \in \nn$. Take a sequence $(b_n)_{n \in \nn}$ with terms in $\nn_{\ge 2}$ such that, for each $n \in \nn$, the equality $s_n + T = b_{n+1} s_{n+1} + T$ holds, and then write $s_n = b_{n+1} s_{n+1} + t_n$ for some $t_n \in T$. Now consider the submonoid
	\[
		M := \langle s_n, t_n \mid n \in \nn \rangle
	\]
	of~$G$. We claim that the monoid $M$ is not atomic, which will contradict the fact that $G$ is hereditarily atomic. First observe that, for each $n \in \nn$, the equality $s_n = b_{n+1} s_{n+1} + t_n$, along with the fact that $t_n \in T$, guarantees the existence of $b'_n \in \nn$ such that $b'_n s_n = b'_n b_{n+1} s_{n+1}$. As a consequence, the sets $\{s_n \mid n \in \nn\}$ and $\uu(M)$ are disjoint if and only if $s_1 \notin \uu(M)$. To argue that $s_1 \notin \uu(M)$, take $c \in \nn_0$ such that $-cs_1 \in M$, and then write
	\begin{equation} \label{eq:checking atomicity}
		-c s_1 = \sum_{n=1}^N c_n s_n + \sum_{n=1}^N d_n t_n
	\end{equation}
	for some $N \in \nn$ and $c_1, \dots, c_N, d_1, \dots, d_N \in \nn_0$. Since $t_1, \dots, t_N$ are torsion elements and the inclusion $s_n + T \subseteq \nn(s_N + T)$ holds for each $n \in \ldb 1,N \rdb$, after multiplying both sides of~\eqref{eq:checking atomicity} by a large positive integer, we obtain that $-c'_0 s_N = \sum_{n=1}^N c'_n s_N$ for some $c'_0, c'_1, \dots, c'_N \in \nn_0$. Because $s_N \notin T$, the fact that $\big( \sum_{n=0}^N c'_n \big) s_N = 0$ ensures that $c'_0 = c'_1 = \cdots = c'_N = 0$, and so $c=0$ in~\eqref{eq:checking atomicity}. As a result, $-s_1 \notin M$, and so $s_1 \notin \uu(M)$. Therefore none of the elements in $\{s_n \mid n \in \nn\}$ is invertible in~$M$. Thus, $M$ is not a group. Moreover, for each $n \in \nn$, the fact that $s_n \in \nn_{\ge 2} s_{n+1} + t_n$ guarantees that $s_n \notin \mathcal{A}(M)$. On the other hand, for each $n \in \nn$, we see that $t_n \notin \mathcal{A}(M)$ because $t_n \in T \cap M \subseteq \uu(M)$. Hence $\mathcal{A}(M)$ is empty, and so the fact that $M$ is not a group ensures that $M$ is not atomic, yielding the desired contradiction.
	\smallskip
	
	(c) $\Rightarrow$ (a): Suppose that the group $G/T$ is cyclic. Assume first that $G/T$ is the trivial group. In this case, $G$ is a torsion group, and it is well known and easy to verify that every submonoid of $G$ must be, in fact, a subgroup of $G$. As every abelian group trivially satisfies the ACCP, it follows that~$G$ is hereditary ACCP.
	
	Assume, therefore, that $G/T$ is not the trivial group. Since $G/T$ is torsion-free, it cannot be finite. Hence $G/T \cong \zz$. Take $g_0 \in G$ such that $G/T = \zz(g_0+T)$ and then, for any $g \in G$, write $g+T \prec T$ (resp., $g+T \succ T$) if $g+T = ng_0 + T$ for some $n \in -\nn$ (resp., $n \in \nn$). We proceed to prove that every submonoid of $G$ satisfies the ACCP. To do so, we fix a submonoid $M$ of~$G$ and split the rest of the proof into the following two cases.
	\smallskip
	
	\textsc{Case 1:} There exist $b_1, b_2 \in M$ such that $b_1 + T \prec T$ and $b_2 + T \succ T$. We will show that~$M$ is a subgroup of $G$. To do so, fix $b \in M$. If $b \in T$, then it is clear that $b \in \uu(M)$. Suppose, therefore, that $b \notin T$. Hence either $b+T \prec T$ or $b+T \succ T$. Assume first that $b + T \prec T$. In this case, we can pick $m, n \in \nn$ such that $b + T = -mg_0 + T$ and $b_2 + T = ng_0 + T$, from which we obtain that  $m(b_2 + T) + n(b + T) = T$. Then $m b_2 + nb \in T$, which implies that $(mN) b_2 + (nN) b = 0$ for some $N \in \nn$ large enough. Hence $b \in \uu(M)$. The case where $b + T \succ T$ is similar. As a result, $M$ is a subgroup of $G$, and so $M$ satisfies the ACCP.
	\smallskip
	
	\textsc{Case 2:} $b + T \preceq T$ for all $b \in M$ or $b + T \succeq T$ for all $b \in M$. After replacing $M$ by its isomorphic copy $-M$ if necessary, we can assume that $b + T \succeq T$ for all $b \in M$. Set
	\[
		\bar{M} := \{b + T \mid b \in M\}.
	\]
	It is clear that $\bar{M}$ is a submonoid of the quotient group $G/T$ contained in $\nn_0(g_0 + T) \cong \nn_0$. Hence~$\bar{M}$ is isomorphic to a numerical monoid, and so it satisfies the ACCP. To prove that $M$ also satisfies the ACCP, suppose that $(b_n + M)_{n \ge 1}$ is an ascending chain of principal ideals of $M$. For each $n \in \nn$, observe that the inclusion $b_n + M \subseteq b_{n+1} + M$ guarantees that $(b_n + T) + \bar{M} \subseteq (b_{n+1} + T) + \bar{M}$. Thus, $((b_n + T) + \bar{M})_{n \ge 1}$ is an ascending chain of principal ideals of $\bar{M}$ and, therefore, it must eventually stabilize. Take $N \in \nn$ such that $(b_n + T) + \bar{M} = (b_N + T) + \bar{M}$ for every $n \in \nn$ with $n \ge N$. Now fix $k \in \nn$ such that $k \ge N$. Since $\bar{M}$ is a numerical monoid up to isomorphism, it is a reduced monoid. Therefore the equality $b_k + T = b_N + T$ holds, and so $b_N = b_k + t_k$ for some $t_k \in T$. On the other hand, the inclusion $b_N + M \subseteq b_k + M$ ensures that $t_k = b_N - b_k \in M$. As a result, $t_k \in M \cap T \subseteq \uu(M)$, which implies that $b_k + M = b_N + M$. Hence the ascending chain of principal ideals $(b_n + M)_{n \ge 1}$ also stabilizes. As a consequence, the submonoid $M$ of $G$ satisfies the ACCP, as desired.
\end{proof}

\medskip
\subsection{A Look Inside Abelian Groups That Are Not Hereditarily Atomic}
\label{subsec:inside abelian groups that are not HA}
Since $\qq$ is a torsion-free abelian group that is not cyclic, Theorem~\ref{thm:HA groups} ensures the existence of an additive submonoid of~$\qq$ that does not satisfy the ACCP: indeed, the additive monoid $\qq_{\ge 0}$ is not even atomic. What is more interesting is that there exist submonoids of~$\qq$ that are atomic but do not satisfy the ACCP. The following examples shed some light upon this observation (another less natural example of an atomic additive submonoid of $\qq$ that does not satisfy the ACCP was constructed by Coykendall and the author in the proof of~\cite[Proposition~5.1]{CG19}).

\begin{ex} \label{ex:atomic PM without ACCP} \hfill
		\begin{enumerate}
		\item[(a)] Let $(p_n)_{n \ge 1}$ be the strictly increasing sequence whose underlying set is the set consisting of all odd primes. Then consider the submonoid $M := \big\langle \frac{1}{2^{n-1} p_n} \mid n \in \nn \big\rangle$\footnote{This monoid is the fundamental ingredient in Grams' construction of the first atomic domain that does not satisfy the ACCP.} of $\qq$. It is not hard to verify that $M$ is an atomic monoid with $\mathcal{A}(M) =  \big\{ \frac{1}{2^{n-1} p_n} \mid n \in \nn \big\}$. However, $M$ does not satisfy the ACCP as, for instance, $\big( \frac{1}{2^n} + M \big)_{n \ge 1}$ is an ascending chain of principal ideals of~$M$ that does not stabilize.
		\smallskip
		
		\item[(b)] Now let $(p_n)_{n \ge 1}$ be the strictly increasing sequence whose underlying set is $\pp$, and consider the submonoid $M := \big\langle \frac{1}{p_n p_{n+2}} \mid n \in \nn \big\rangle$ of $\qq$. As in the previous example, it is not difficult to verify that $M$ is atomic with $\mathcal{A}(M) = \big\{ \frac{1}{p_n p_{n+2}} \mid n \in \nn \big\}$. On the other hand, $M$ does not satisfy the ACCP. Indeed, since
		\[
			\Big( \frac 1{p_n} + \frac 1{p_{n+1}} \Big) - \Big( \frac 1{p_{n+1}} + \frac 1{p_{n+2}} \Big) = (p_{n+2} - p_n) \frac 1{p_n p_{n+2}} \in M
		\]
		for every $n \in \nn$, the sequence $\big(  \frac 1{p_n} + \frac 1{p_{n+1}} + M \big)_{n \ge 1}$ is an ascending chain of principal ideals of~$M$ that does not stabilize.
		\smallskip
		
		\item[(c)] Let $q$ be a positive rational such that $0 < q < 1$ and $q^{-1} \notin \nn$, and write $q = \frac{a}{b}$ with $a,b \in \nn$ and $\gcd(a,b) = 1$. Now consider the submonoid $M := \langle q^n \mid n \in \nn_0 \rangle$ of $\qq$. It is not difficult to verify that $M$ is atomic (for this, we need the condition $q^{-1} \notin \nn$) with $\mathcal{A}(M) =  \{ q^n \mid n \in \nn_0 \}$. However, as the equality $a q^n = (b-a) q^{n+1} + a q^{n+1}$ holds for every $n \in \nn$, we infer that $\big( a q^n + M \big)_{n \ge 1}$ is an ascending chain of principal ideals of $M$ that does not stabilize. Thus,~$M$ does not satisfy the ACCP.
	\end{enumerate}
\end{ex}

As it is the case with $\qq$, the group $\zz^2$ is a torsion-free abelian group that is not cyclic. Hence, in light of Theorem~\ref{thm:HA groups}, there exists a submonoid of $\zz^2$ that does not satisfy the ACCP: indeed, we have seen in Example~\ref{ex:Z^2 is not HA} that the nonnegative cone of $\zz^2$ with respect to the lexicographic order (with priority on the first coordinate) is a submonoid of $\zz^2$ that is not even atomic. A more subtle question is that of whether $\zz^2$ has a submonoid that is atomic but does not satisfy the ACCP. The rest of this section is devoted to a construction that will provide a positive answer to this question. We start with the following needed lemmas.

\begin{lem} \label{lem:lines and lattice points}
	Let $L$ be a real line through the origin in $\rr^2$. Then for every $\epsilon > 0$ there exists $v \in \zz^2 \setminus \{(0,0)\}$ such that $d(v,L) < \epsilon$.
\end{lem}

\begin{proof}
	Fix $\epsilon > 0$. If the slope of $L$ is a rational number or infinite, then it is clear that $L$ must contain a nonzero vector with rational coordinates, and so the statement of the lemma follows. Then we assume that the slope of~$L$ is irrational. Since the slope of $L$ is not $0$, we can pick $\alpha \in \rr$ such that $w := (\alpha,1) \in L$. Observe that $\alpha$ is irrational because the slope of $L$ is irrational. It is well known (and not difficult to verify) that the set $\{n \alpha - \lfloor n \alpha \rfloor \mid n \in \nn\}$ is dense in the interval $[0,1]$. Then there exists $m \in \nn$ such that $m \alpha - \lfloor m \alpha \rfloor < \epsilon$. Setting $v = (\lfloor m \alpha \rfloor, m)$ and taking into account the fact that the points in $L$ are described by the equation $x - \alpha y = 0$, we obtain that
	\[
		d(v,L) = \frac{|x - \alpha y|_{(x,y) = v}}{\sqrt{1 + \alpha^2 }} = \frac{ m\alpha - \lfloor m\alpha \rfloor }{\sqrt{1 + \alpha^2 }} < \epsilon.
	\]
	Hence $v$ is a nonzero lattice point satisfying that $d(v,L) < \epsilon$.
\end{proof}

\begin{lem} \label{lem:line with no lattice point}
	There exists a line that is tangent to the unit circle and does not contain any lattice point.
\end{lem}

\begin{proof}
	Let $\mathcal L$ be the set consisting of all the real lines in $\rr^2$ with irrational slopes that are tangent to the unit circle at tangent points having positive $y$-coordinates (i.e, tangent to the upper unit circle). There is a natural bijection between $\mathcal L$ and $\rr \setminus \qq$ and, therefore,~$\mathcal L$ is uncountable. Observe that no line $\ell \in \mathcal L$ can contain two distinct lattice points as, otherwise, $\ell$ would have a rational slope. Suppose, by way of contradiction, that every line $\ell \in \mathcal L$ contains at least one lattice point $P_\ell$, and so exactly one lattice point. Hence the assignments $\ell \mapsto P_\ell$ (for all $\ell \in \mathcal{L}$) induce a function $f \colon \mathcal L \to \zz^2$ satisfying that $|f^{-1}(P)| \le 2$ for all $P \in \zz^2$. Thus, the fact that $\zz^2$ is countable, in tandem with the equality $\mathcal L = \bigcup_{P \in \zz} f^{-1}(P)$, would imply that $\mathcal L$ is countable, which is a contradiction. As a result, there is a line in $\mathcal L$ that does not intersect $\zz^2$.
\end{proof}

We proceed to construct, for each $d \in \nn_{\ge 2}$, a rank-$d$ lattice monoid inside $\zz^d$ that is atomic but does not satisfy the ACCP.

\begin{prop}  \label{prop:atomic monoid that is not ACCP}
	For each $d \in \nn_{\ge 2}$, there exists a reduced rank-$d$ submonoid of $\zz^d$ that is atomic but does not satisfy the ACCP.
\end{prop}

\begin{proof}
	We first prove the statement of the proposition for the case $d=2$. Once we have produced an atomic submonoid of $\zz^2$ that does not satisfy the ACCP, the rest will follow easily.
	
	By virtue of Lemma~\ref{lem:line with no lattice point}, there exists a line $L$ in $\rr^2$ that is tangent to the upper-half of the unit circle and does not intersect $\zz^2$. By using symmetry we can further assume that the tangent point of~$L$ and the unit circle belongs to the second quadrant, and so that~$L$ has a positive slope. Let~$u$ be the unit vector determined by this tangent point, and let $v$ be the unit vector that is orthogonal to~$u$ and lies in the first quadrant. Now set $L_0 := \rr v$. We will construct a lattice monoid $M$ contained in the upper half-space $L_0^+$ determined by $L_0$. Let $p_u$ and $p_v$ denote the canonical projections of $\rr^2$ onto the one-dimensional subspaces $\rr u$ and $\rr v$, respectively.
	\smallskip
	
	First, we construct inductively a sequence $(a_n)_{n \ge 0}$ with terms in $\zz^2 \cap L_0^+$ satisfying certain conditions, and then we will show that the lattice monoid generated by the set $\{a_n \mid n \in \nn_0\}$ is atomic but does not satisfy the ACCP. Take $a_0 = (0,1) \in \zz^2 \cap L_0^+$. Notice that $a_0 \in L^-$. Set
	\[
		m_0 :=  \min\{m \in \nn \mid m a_0 \in L^+\}.
	\]
	Observe that $d(m_0 a_0, L) > 0$ because $L$ does not contain any lattice point. By Lemma~\ref{lem:lines and lattice points}, we can take $a_2 \in \zz^2$ in such a way that $d(a_2,L_0) < \frac 12 d(m_0a_0, L)$. After replacing~$a_2$ by $-a_2$ if necessary, we can assume that $a_2 \in \zz^2 \cap L_0^+$. Now we set
	\[
		m_2 :=  \min\{m \in \nn \mid m a_2 \in L^+\} \quad \text{ and } \quad a_1 := m_0a_0 - m_2a_2.
	\]
	Figure~\ref{fig:AtomicLatticeMonoidNotACCP} illustrates the vectors and affine lines we have introduced so far.
	\begin{figure}[h]
		\includegraphics[width=8cm]{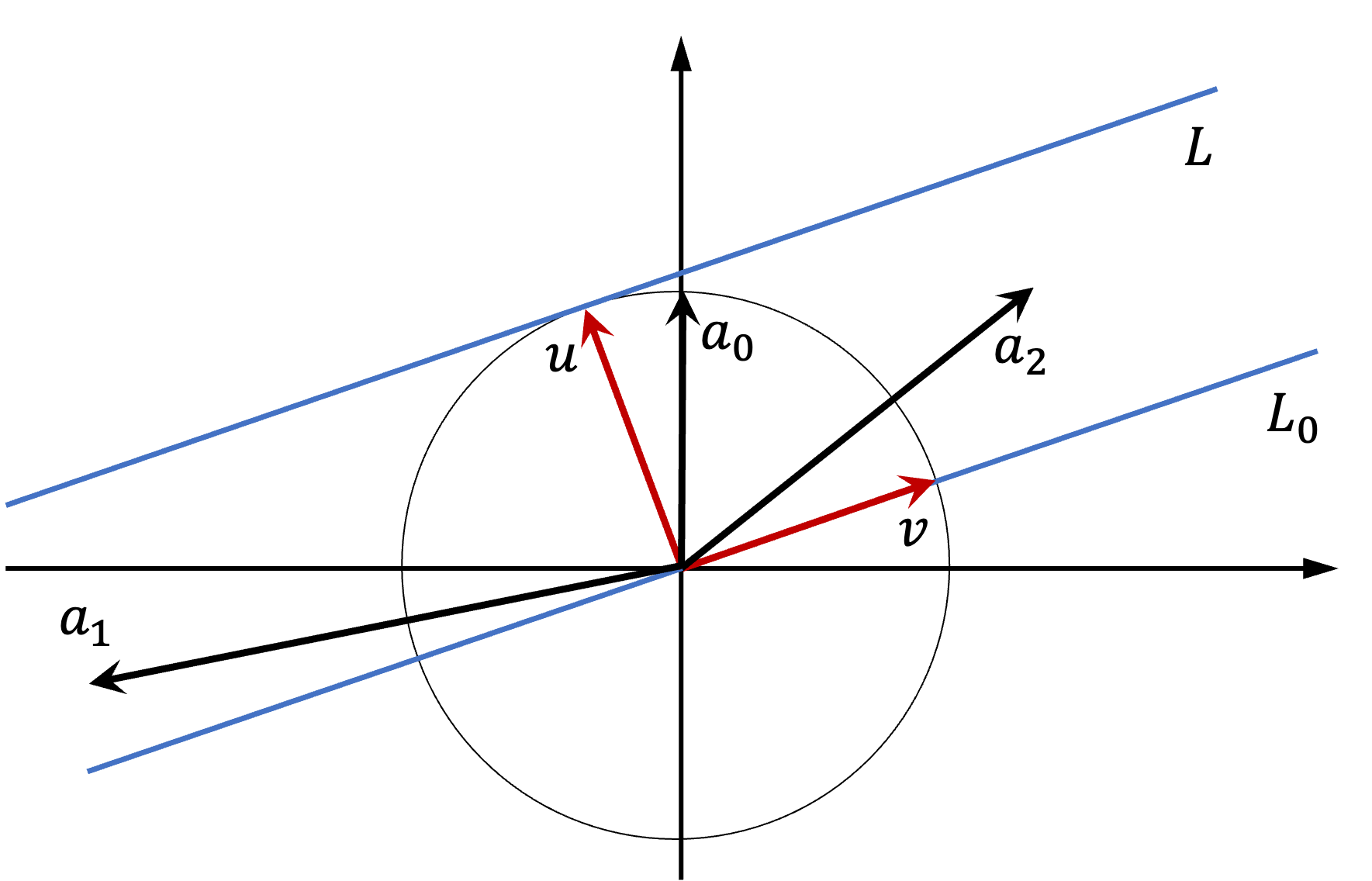}
		\caption{Vectors $a_0, a_1, a_2$ in the base case of the inductive construction of the sequence $(a_n)_{n \ge 1}$.}
		\label{fig:AtomicLatticeMonoidNotACCP}
	\end{figure}
	The minimality of $m_2$ guarantees that $d(m_2 a_2, L) < d(a_2, L_0)$. Then $d(m_2 a_2, L) < d(m_0a_0, L)$, from which we can deduce that $m_0a_0 - m_2a_2 \in L_0^+ \setminus L_0$ and, therefore, $a_1 \in \zz^2 \cap L_0^+$ and $d(a_1, L_0) > 0$. In addition, we can check that
	\begin{equation} \label{eq:non-ACCP 1}
		d(a_1, L_0) = \norm{p_u(m_0 a_0 - m_2 a_2)} = d(m_0 a_0, L) - d(m_2 a_2, L) < d(m_0 a_0, L) < d(a_0,L_0),
	\end{equation}
	where the last inequality follows from the minimality of $m_0$. Also, because $d(a_2,L_0) < \frac 12 d(m_0a_0, L)$ and $m_0a_0 = a_1 + m_2 a_2$, we see that
	\begin{align} \label{eq:non-ACCP 2}
		2d(a_2, L_0) &< d(m_0 a_0, L) = d(m_0 a_0, L_0) - 1 = \norm{p_u(a_1 + m_2 a_2)} - 1 \\
							\label{eq:non-ACCP 3}   &=\norm{p_u(a_1)} + (\norm{p_u(m_2 a_2)} - 1) = d(a_1, L_0) + d(m_2 a_2, L) < d(a_1, L_0) + d(a_2,L_0),
	\end{align}
	where the last inequality follows from the minimality of $m_2$. Thus, $d(a_2, L_0) < d(a_1, L_0) < d(a_0,L_0)$. Now consider the monoid $M_2 := \langle a_0, a_1, a_2 \rangle$. As $d(a_2, L_0) < d(a_0, L_0)$, from the fact $a_2 \in \zz^2 \cap L_0^+$ we infer that $a_2 \notin \rr a_0$. Therefore $\{a_0, a_2\}$ is a basis for $\rr^2$ and, as a result, the equality $a_1 = m_0a_0 - m_2a_2$ guarantees that $a_1 \notin \langle a_0, a_2 \rangle$. Hence $a_1 \in \mathcal{A}(M_2)$. Similarly, we can verify that $a_2 \in \mathcal{A}(M_2)$. Now observe that because $\{a_0, a_2\}$ is a basis for $\rr^2$, the equality $m_0a_0 = a_1 + m_2 a_2$ ensures that $\{a_1, a_2\}$ is also a basis for $\rr^2$. Thus, the inequality $m_0 \ge 2$, in tandem with the equality $a_0 = \frac{1}{m_0}a_1 + \frac{m_2}{m_0} a_2$, guarantees that $a_0 \notin \langle a_1, a_2 \rangle$. This implies that $a_0 \in \mathcal{A}(M_2)$. Hence $\mathcal{A}(M_2) = \{a_0, a_1, a_2\}$.
	
	Now suppose that, for some $n \in \nn$, we have constructed a finite sequence $(a_k)_{k \in \ldb 0, 2n \rdb}$ of lattice points in $\zz^2 \cap L_0^+$ and a finite sequence $(m_{2k})_{k \in \ldb 0,n \rdb}$ of positive integers in $\nn_{\ge 2}$ such that
	the following conditions hold:
	 \begin{enumerate}
	 	\item $a_{2k+1} = m_{2k} a_{2k} - m_{2k+2} a_{2k+2}$ for every $k \in \ldb 0, n-1 \rdb$,
	 	\smallskip
	 	
	 	\item $m_{2k+2} = \min\{m \in \nn \mid m  a_{2k+2} \in L^+ \}$ for every $k \in \ldb 0, n-1 \rdb$,
	 	\smallskip
	 	
	 	\item $d(a_{k+1}, L_0) < d(a_k, L_0)$ for every $k \in \ldb 0, 2n-1 \rdb$, and 
	 	\smallskip
	 	
	 	\item $M_{2n} := \langle a_k \mid k \in \ldb 0,2n \rdb \rangle$ is an atomic monoid with $\mathcal{A}(M_{2n}) := \{a_k \mid k \in \ldb 0, 2n \rdb \}$.
	 \end{enumerate}
 	
 	\smallskip
 	 \noindent {\it Claim 1.} There exists $\ell \in \nn$ such that, for any $w \in L_0^+ \cap L^-$, the inequality $\norm{p_v(w)} > \ell$ implies that $w + a_i \notin \text{cone}(M_{2n})$ for any $i \in \ldb 0, 2n \rdb$.
 	\smallskip
 	
 	\noindent {\it Proof of Claim 1.} Take $\ell \in \nn$ large enough so that
 	\[
 		4 \frac{\max \{ \norm{p_v(a_i)} \mid i \in \ldb 0, 2n \rdb \}}{ \min \{ \norm{p_u(a_i)} \mid i \in \ldb 0,2n \rdb \}} < \ell.
 	\]
 	Now assume, towards a contradiction, that there exists $w \in L_0^+ \cap L^-$ with $\norm{p_v(w)} > \ell$ such that $w + a_j \in \text{cone}(M_{2n})$ for some $j \in \ldb 0, 2n \rdb$. Write $w + a_j = \sum_{i=0}^{2n} r_i a_i$ for some $r_0, \dots, r_{2n} \in \rr_{\ge 0}$, and observe that
	\[
		 \min \big\{ \norm{p_u(a_i)} \mid i \in \ldb 0,2n \rdb \big\} \sum_{i=0}^{2n} r_i \le  \sum_{i=0}^{2n} r_i \norm{ p_u(a_i) } = \norm{ p_u\bigg( \sum_{i=0}^{2n} r_i a_i \bigg)} = \norm{  p_u(w + a_j) } < 2.
 	\]
 	Therefore $\sum_{i=0}^{2n} r_i < \frac 2{  \min \{ \norm{p_u(a_i)} \mid i \in \ldb 0,2n \rdb \}  }$, and we can use this inequality to infer that
 	\begin{equation} \label{eq:lastone0}
 		2 \norm{p_v(w + a_j)} = 2 \sum_{i=0}^{2n} r_i \norm{ p_v(a_i)} \le 2 \max \big\{\norm{p_v(a_i)} \mid i \in \ldb 0, 2n \rdb \big\}  \sum_{i=0}^{2n} r_i < \ell.
 	\end{equation}
 	In addition, we see that $\norm{p_v(w)} > \ell > 2 \max \{ \norm{p_v(a_i)} \mid i \in \ldb 0, 2n \rdb \} \ge 2 \norm{ p_v(a_j)}$, and using these inequalities we can verify that
 	\begin{equation} \label{eq:lastone1}
 		\norm{p_v(w)} < \norm{p_v(w)} + \big( \norm{ p_v(w)} - 2 \norm{ p_v(a_j)} \big) \le 2 \norm{ p_v(w + a_j)}. 
 	\end{equation}
 	Putting together the inequalities in~\eqref{eq:lastone0} and~\eqref{eq:lastone1}, we obtain that $\norm{p_v(w)} < \ell$, which contradicts our initial assumption. Hence we have established Claim~1.
 	\smallskip
	 
	 Let $\ell$ be as in Claim~1. Because $m_{2n} a_{2n} \notin L$, Lemma~\ref{lem:lines and lattice points} allows us to take $a_{2n+2} \in \zz^2$ in such a way that $d(a_{2n+2}, L_0) < \frac 12 d(m_{2n} a_{2n}, L)$. After replacing $a_{2n+2}$ by $-a_{2n+2}$ if necessary, we can assume that $a_{2n+2} \in \zz^2 \cap L_0^+ \cap L^-$. Also, we can assume that $\norm{a_{2n+2}}$ is sufficiently large so that
	 \begin{equation} \label{eq:aux bound}
	 	\norm{p_v(a_{2n+2})} > \max \big\{ \ell, \norm{p_v(m_{2n} a_{2n})} \big\}.
	 \end{equation}
 	Now set
 	\[
 		m_{2n+2} := \min \{ m \in \nn \mid m a_{2n+2} \in L^+ \} \quad \text{and} \quad a_{2n+1} = m_{2n} a_{2n} - m_{2n+2} a_{2n+2} 
 	\]
	 to extend conditions~(1) and~(2) above to the interval $\ldb 0,n \rdb$. As $L$ does not contain any lattice points, $a_{2n+2} \in L^- \setminus L$, whence $m_{2n+2} \ge 2$. In light of the inequality $d(a_{2n+2}, L_0) < \frac 12 d(m_{2n} a_{2n}, L)$ and the minimality of both $m_{2n}$ and $m_{2n+2}$ (here we are using condition~(2) above for $k = n-1$), we can proceed as we did in~\eqref{eq:non-ACCP 1} and~\eqref{eq:non-ACCP 2}--\eqref{eq:non-ACCP 3} to verify that $ d(a_{2n+2}, L_0) < d(a_{2n+1}, L_0) < d(a_{2n}, L_0)$, which extends condition~(3) above to the interval $\ldb 0, 2n+1 \rdb$.
	 In addition,
	 \begin{align*}
	 	\norm{p_v(a_{2n+1})} &= \norm{p_v(m_{2n} a_{2n}) - p_v(m_{2n+2} a_{2n+2})} \\
	 										  &\ge \norm{p_v(m_{2n+2} a_{2n+2})} - \norm{p_v(m_{2n} a_{2n})} >  \norm{p_v(a_{2n+2})},
	 \end{align*}
	 where the last inequality is obtained from the facts that $m_{2n+2} \ge 2$ and $\norm{p_v(a_{2n+2})} > \norm{p_v(m_{2n} a_{2n})}$ (by~\eqref{eq:aux bound}). Therefore the following inequality holds:
	 \begin{equation} \label{eq:non-ACCP 3'}
	 	\norm{p_v(a_{2n+1})} > \ell > 4 \frac{\max \{ \norm{p_v(a_i)} \mid i \in \ldb 0, 2n \rdb \}}{ \min \{ \norm{p_u(a_i)} \mid i \in \ldb 0,2n \rdb \}}.
	 \end{equation}
	 In order to complete our inductive construction of the sequences $(a_n)_{n \ge 0}$ and $(m_{2n})_{n \ge 0}$, we only need to extend condition~(4) above to $2n+2$; that is, we need to argue the following claim for the monoid $M_{2n+2} := \langle a_i \mid i \in \ldb 0, 2n+2 \rdb \rangle$.
	 \smallskip
	 
	 \noindent {\it Claim 2.} $\mathcal{A}(M_{2n+2}) = \big\{ a_i \mid i \in \ldb 0, 2n+2 \rdb \big\}$.
	 \smallskip
	 
	 \noindent {\it Proof of Claim 2.} We will show that $a_{2n+1} \notin \text{cone}(M_{2n} \cup \{a_{2n+2}\})$ and $a_{2n+2} \notin \text{cone}(M_{2n} \cup \{a_{2n+1}\})$. Suppose, towards a contradiction, that $a_{2n+1} \in \text{cone}(M_{2n} \cup \{a_{2n+2}\})$, and write $a_{2n+1} = b + r a_{2n+2}$ for some $b \in \text{cone}(M_{2n})$ and $r \in \rr_{\ge 0}$. Because $a_{2n+1} = m_{2n} a_{2n} - m_{2n+2} a_{2n+2}$,
	 \[
	 	\Big( 1 + \frac{r}{m_{2n+2}} \Big) a_{2n+1} = b + \frac{r m_{2n}}{m_{2n+2}} a_{2n} \in \text{cone}(M_{2n}),
	 \]
	 which implies that $a_{2n+1} \in \text{cone}(M_{2n})$. Thus, $a_{2n+1} + a_1 \in \text{cone}(M_{2n})$, and so it follows from Claim~1 that  $\norm{p_v(a_{2n+1})} \le \ell$. However, this contradicts~\eqref{eq:non-ACCP 3'}. Hence $a_{2n+1} \notin \text{cone}(M_{2n} \cup \{a_{2n+2}\})$. In a completely similar way, we can check that $a_{2n+2} \notin \text{cone}(M_{2n} \cup \{a_{2n+1}\})$. As a consequence, $a_{2n+1}, a_{2n+2} \in \mathcal{A}(M_{2n+2})$. 
	 
	 We proceed to argue that $a_{2n+2}$ (resp., $a_{2n+1}$) does not divide any of the elements $a_0, \dots, a_{2n}$ in the monoid $\langle M_{2n} \cup \{a_{2n+2}\} \rangle$ (resp., $\langle M_{2n} \cup \{a_{2n+1}\} \rangle$). Suppose first, by way of contradiction, that $a_{2n+2}$ divides an element $a \in \{a_i \mid i \in \ldb 0, 2n \rdb \}$ in $\langle M_{2n} \cup \{a_{2n+2}\} \rangle$, and write $a = c a_{2n+2} + \sum_{i=1}^m b_i$ for some $c \in \nn$ and $b_1, \dots, b_m \in \{a_i \mid i \in \ldb 0, 2n \rdb \}$. Then we see that
	 \begin{equation} \label{eq:aux 5}
	 	c \norm{p_v(a_{2n+2})} = \norm{p_v(a) - \sum_{i=1}^m p_v(b_i)} \le (m +1) \max \big\{ \norm{p_v(a_i)} \mid i \in \ldb 0, 2n \rdb \big\}.
	 \end{equation}
	 Therefore
	 \begin{align} \label{eq:non-ACCP 4}
	 	1 + d(a,L_0) &= 1 + \norm{ p_u(a) } > 1 + \sum_{i=1}^m \norm{p_u(b_i) } > (m + 1) \norm{ p_u(a_{2n}) } \\
	 						 &\ge  \frac{ \norm{p_v(a_{2n+2})} \cdot \norm{ p_u(a_{2n}) }}{  \max\{ \norm{p_v(a_i)} \mid i \in \ldb 0, 2n \rdb \} } >   \frac{ 4 \, \norm{ p_u(a_{2n}) }}{  \min \{ \norm{p_u(a_i)} \mid i \in \ldb 0,2n \rdb \} } = 4,
	 \end{align}
	 where the second inequality follows from condition~(3) above (because $\norm{p_u(a_i)} = d(a_i, L_0)$ for every $i \in \ldb 0, 2n \rdb$), the third inequality follows from~\eqref{eq:aux 5}, and the fourth inequality follows from~\eqref{eq:aux bound}. As a result, we obtain that $d(a, L_0) > 1$, which is a contradiction. Hence $a_{2n+2}$ does not divide any of the elements $a_0, \dots, a_{2n}$ in the monoid $\langle M_{2n} \cup \{a_{2n+2}\} \rangle$. Similarly, but using~\eqref{eq:non-ACCP 3'} instead of~\eqref{eq:aux bound} at the end of the argument, we can prove that $a_{2n+1}$ does not divide any of the elements $a_0, \dots, a_{2n}$ in the monoid $\langle M_{2n} \cup \{a_{2n+1}\} \rangle$.	
	 
	 Let us argue now that neither $a_{2n+1}$ nor $a_{2n+2}$ divide any of the elements $a_0, \dots, a_{2n}$ in the larger monoid $M_{2n+2}$. Assume, for the sake of a contradiction, that one of the atoms $a_{2n+1}$ or $a_{2n+2}$ divides an element $a \in \{a_i \mid i \in \ldb 0, 2n \rdb \}$ in $M_{2n+2}$. Since $a_{2n+2}$ does not divide any of the elements $a_0, \dots, a_{2n}$ in the monoid $\langle M_{2n} \cup \{a_{2n+2}\} \rangle$ and $a_{2n+1}$ does not divide any of the elements $a_0, \dots, a_{2n}$ in the monoid $\langle M_{2n} \cup \{a_{2n+1}\} \rangle$, we can write $a = \sum_{i=0}^{2n+2} c_i a_i$ for some $c_0, \dots, c_{2n+2} \in \nn_0$ such that $c_{2n+1} > 0$ and $c_{2n+2} > 0$. Now, after replacing $a_{2n+1}$ by $m_{2n} a_{2n} - m_{2n+2} a_{2n+2}$ in $a = \sum_{i=0}^{2n+2} c_i a_i$, we obtain the following:
	 \begin{equation} \label{eq:aux 6}
	 	(c_{2n+1} m_{2n+2} - c_{2n+2}) a_{2n+2} + a = (c_{2n+1} m_{2n} + c_{2n}) a_{2n} + \sum_{i=0}^{2n-1} c_i a_i.
	 \end{equation}
	Thus, $(c_{2n+1} m_{2n+2} - c_{2n+2}) a_{2n+2} + a \in \text{cone}(M_{2n})$. Since $a \in \mathcal{A}(M_{2n})$ and $a_{2n+2}$ does not divide~$a$ in $\langle M_{2n} \cup \{a_{2n+2}\} \rangle$, it follows from~\eqref{eq:aux 6} that the coefficient $c_{2n+1} m_{2n+2} - c_{2n+2}$ is positive. Therefore $(c_{2n+1} m_{2n+2} - c_{2n+2}) (a_{2n+2} + a) \in \text{cone}(M_{2n})$, which implies that $a_{2n+2} + a \in \text{cone}(M_{2n})$. However, this is a contradiction as, by Claim~1, the fact that $\norm{p_v(a_{2n+2})} > \ell$ implies that $a_{2n+2} + a \notin \text{cone}(M_{2n})$. Hence neither $a_{2n+1}$ nor $a_{2n+2}$ divide any of the elements $a_0, \dots, a_{2n}$ in the monoid $M_{2n+2}$.
	
	Since $\mathcal{A}(M_{2n}) = \{a_i \mid i \in \ldb 0,2n \rdb \}$, the statement proved in the previous paragraph guarantees that $ \{a_i \mid i \in \ldb 0,2n \rdb \} \subseteq \mathcal{A}(M_{2n+2})$. This, along with the fact that both $a_{2n+1}$ and $a_{2n+2}$ belong to $\mathcal{A}(M_{2n+2})$, allows us to conclude that $\mathcal{A}(M_{2n+2}) =  \{a_i \mid i \in \ldb 0,2n+2 \rdb \}$, which completes the proof of Claim~2.
	\smallskip
	
	Thus, there exists a sequence $(a_n)_{n \ge 0}$ of lattice points in $L_0^+$ satisfying the conditions (1)--(4) above. Now consider the lattice monoid $M := \bigcup_{n \in \nn_0} M_{2n}$. Observe that $M = \langle a_n \mid n \in \nn_0 \rangle$. Moreover, for each $k \in \nn_0$, the fact that $a_k \in \mathcal{A}(M_{2n})$ for all $n \in \nn_0$ with $k \le 2n$ immediately implies that $a_k \in \mathcal{A}(M)$. As a result, $\mathcal{A}(M) = \{a_n \mid n \in \nn_0 \}$, and so $M$ is atomic. To argue that $M$ does not satisfy the ACCP, consider the sequence $(m_{2n} a_{2n} + M)_{n \ge 0}$ of principal ideals of $M$. It follows from condition~(1) above that such a sequence is an ascending chain of principal ideals. On the other hand, for every $n \in \nn_0$ the inclusion $m_{2n} a_{2n} + M \subseteq m_{2n+2} a_{2n+2} + M$ is strict as $M$ is a reduced monoid and $m_{2n+2} a_{2n+2}$ properly divides $m_{2n} a_{2n}$ in~$M$. Hence $M$ contains an ascending chain of principal ideals that does not stabilize and, as a consequence, $M$ does not satisfy the ACCP.
	\smallskip
	
	Finally, fix $d \in \nn$ with $d \ge 2$ and consider the lattice submonoid $M_d := M \times \nn_0^{d-2}$ of $\zz^d$. It is clear that $\text{rank} \, M_d = d$. In addition, as both $M$ and $\nn_0^{d-2}$ are atomic, $M_d$ must be atomic. On the other hand, observe that $M$ is a divisor-closed submonoid of $M_d$. This, along with the fact that $M$ does not satisfy the ACCP, guarantees that $M_d$ does not satisfy the ACCP neither, which concludes the proof.
\end{proof}

It follows from~\cite[Theorem~3.1]{GV23} that every reduced torsion-free monoid is hereditarily atomic if and only if it satisfies the ACCP. Since the lattice monoids $M_d$ constructed in the proof of Proposition~\ref{prop:atomic monoid that is not ACCP} are reduced, we obtain the following corollary.

\begin{cor} \label{cor:atomic lattice monoid not ACCP}
	For each $d \in \nn_{\ge 2}$, there exists a rank-$d$ submonoid of $\zz^d$ that is atomic but not hereditarily atomic.
\end{cor}

In the same direction of Corollary~\ref{cor:atomic lattice monoid not ACCP}, J. Correa-Morris and the author in~\cite[Proposition~4.10]{CG22} constructed, for each $d \in \nn$, a class of rank-$d$ atomic monoids that do not satisfy the ACCP. However, the monoids in Corollary~\ref{cor:atomic lattice monoid not ACCP} seem to be the first known atomic monoids that do not satisfy the ACCP whose difference groups are free.

\bigskip
\section{Group Algebras}
\label{sec:commutative group rings}

In this section, we consider hereditary atomicity in group algebras. Our primary purpose is to characterize the group algebras that are hereditarily atomic, which we will do in Subsection~\ref{subsec:HA group algebras}. In Subsection~\ref{subsec:inside group algebras that are not HA}, we first take a look inside group algebras that are not hereditarily atomic, and then we compare hereditary atomicity with the property of satisfying the ACCP.

\smallskip
\subsection{Hereditarily Atomic Group Algebras}
\label{subsec:HA group algebras}

To prove the following theorem, which is the primary result of this section, we will apply the characterization of hereditarily atomic abelian groups established in Theorem~\ref{thm:HA groups}.

\smallskip
\begin{prop}\label{prop:group rings}
	Let $R$ be an integral domain, and let $G$ be a nontrivial abelian group. Then the following conditions are equivalent.
	\begin{enumerate}
%
		\item[(a)] $R[G]$ is hereditarily atomic.
		\smallskip
		
		\item[(b)] $R$ is an algebraic extension of $\ff_p$ for some $p \in \pp$ and $G$ is the infinite cyclic group.
	\end{enumerate}
\end{prop}

\begin{proof}
	
	(a) $\Rightarrow$ (b): Suppose first that $R[G]$ is hereditarily atomic. Note that $G$ must be torsion-free because the existence of a nonzero element $g \in G$ with $ng = 0$ for some $n \in \nn_{\ge 2}$ would imply that
	\[
		(x-1)(x^{(n-1)g} + \dots + x + 1) = x^{ng} - 1 = 0,
	\]
	which is not possible because $R[G]$ is atomic and so an integral domain. Because $G$ is nontrivial and torsion-free, it must contain a subgroup isomorphic to $\zz$. Therefore $R[G]$ contains a subring isomorphic to the ring of Laurent polynomials $R[x^{\pm 1}]$. Since $R[G]$ is hereditarily atomic, $R[x^{\pm 1}]$ is also hereditarily atomic. Thus, it follows from \cite[Theorem~5.10]{CGH21} that $R$ is an algebraic extension of the field $\ff_p$ for some $p \in \pp$ (and so $R$ is a field). 
	\smallskip
	
	Before we can conclude that $G$ is an infinite cyclic group, we verify that $G$ is hereditarily atomic. To do so, suppose that $M$ is a submonoid of $G$, and consider the monoid algebra $R[M]$. Since $R[M]$ is a subring of $R[G]$, the former must be atomic domain. Thus, it follows from~\cite[Proposition~1.4]{hK01} that $M$ is atomic. Therefore~$G$ is hereditarily atomic, as claimed. Since $G$ is hereditarily atomic and torsion-free, it follows from Theorem~\ref{thm:HA groups} that $G$ is cyclic. Finally, the fact that $G$ is nontrivial and torsion-free guarantees that~$G$ is the infinite cyclic group.

	(b) $\Rightarrow$ (a): Suppose now that $R$ is an algebraic extension of the field $\ff_p$ for some $p \in \pp$ and also that~$G$ is the infinite cyclic group. 
	Since $G \cong \zz$, the group algebra $R[G]$ is isomorphic to the ring of Laurent polynomials over~$R$. Thus, it follows from \cite[Theorem~5.10]{CGH21} that $R[G]$ is hereditarily atomic, which concludes the proof. 
\end{proof}

In light of Proposition~\ref{prop:group rings}, every hereditarily atomic group algebra satisfies the ACCP.

\begin{cor} \label{cor:HA group domains are ACCP}
	Every hereditarily atomic group algebra satisfies the ACCP.
\end{cor}

\begin{proof}
	By virtue of Proposition~\ref{prop:group rings}, every hereditarily atomic group algebra is, up to isomorphism, a Laurent polynomial ring $F[x^{\pm 1}]$ with $F$ being an algebraic field extension of $\ff_p$ for some $p \in \pp$. Since $F[x^{\pm 1}]$ is a UFD, it must satisfy the ACCP.
\end{proof}

As the following example illustrates, the converse of Corollary~\ref{cor:HA group domains are ACCP} does not hold.

\begin{ex} \label{ex:group algebra ACCP but not HA}
	Let $F$ be a field that is not an algebraic extension of $\ff_p$ for any $p \in \pp$, and consider the group algebra $F[\zz]$, which can be identified with the Laurent polynomial ring $F[x^{\pm 1}]$. Since $F[x^{\pm 1}]$ is a UFD (or a Noetherian ring), it must satisfy the ACCP. However, it follows from Proposition~\ref{prop:group rings} that the group algebra $F[x^{\pm 1}]$ is not hereditarily atomic: indeed, we can directly see that $F[x^{\pm 1}]$ contains the non-atomic subring $\zz + x\qq[x]$ if $F$ has characteristic zero or the non-atomic subring $\ff_p[t] + x \ff_p(t)[x]$ if~$F$ has characteristic $p \in \pp$, where $t \in F$ is a transcendental element over $\ff_p$. In particular, not every atomic group algebra is hereditarily atomic.
\end{ex}

In the more general class of integral domains, we still do not know whether every hereditarily atomic domain satisfies the ACCP. The following question is a version of \cite[Conjecture~6.1]{CGH21}.

\begin{question}
	Does every hereditarily atomic domain satisfy the ACCP?
\end{question}

Being a particular case of the group algebra in Example~\ref{ex:group algebra ACCP but not HA}, the group algebra $\qq[\zz]$ is an atomic domain that is not hereditarily atomic. Moreover, from the fact that the unique factorization property ascends from an integral domain to its Laurent polynomial extensions in finitely many variables, one can readily deduce that for each $k \in \nn_{\ge 2}$, the group algebra $F[\zz^k]$ is atomic, but it is not hereditarily atomic by virtue of Proposition~\ref{prop:group rings}. There are, on the other hand, group algebras (over fields) that are not even atomic. The next proposition illustrates the latter observation.

Recall that a field $F$ is perfect provided that each irreducible polynomial over $F$ has distinct roots. It is well known that a field $F$ is perfect if and only if either $F$ has characteristic zero or $F$ has prime characteristic $p$ and every element of $F$ can be written as a $p$-th power in $F$. In addition, recall that a field $F$ is called a real closed field if $F$ is not algebraically closed, but its field extension $F(\sqrt{-1})$ is algebraically closed. It is not hard to verify that every irreducible polynomial over a real closed field has degree at most~$2$. We are in a position to prove the following proposition, which yields classes of antimatter group algebras.

\begin{prop}  \label{prop:antimatter group algebras} 
	For a field $F$ and a torsion-free abelian group $G$, the following statements hold.
	\begin{enumerate}
		\item If $F$ is perfect of characteristic $p \in \pp$ and $G$ is $p$-divisible, then $F[G]$ is antimatter.
		\smallskip
		
		\item If $F$ is real closed and $G$ is non-cyclic with rank $1$, then $F[G]$ is antimatter.
		\smallskip
		
		\item If $F$ is algebraically closed and $G$ is non-cyclic with rank $1$, then $F[G]$ is antimatter.
	\end{enumerate}
\end{prop}

\begin{proof}
	(1) Fix $p \in \pp$ such that $F$ is a perfect field of characteristic $p$ and $G$ is a $p$-divisible abelian group. To argue that $F[G]$ is antimatter, fix a nonzero nonunit $f \in F[G]$, and then write $f$ as $f = \alpha_n x^{g_n} + \dots + \alpha_1 x^{g_1}$ for some nonzero coefficients $\alpha_1, \dots, \alpha_n \in F$ and some distinct exponents $g_1, \dots, g_n \in G$. Since $F$ is a perfect field of characteristic $p$, we can pick $\beta_1, \dots, \beta_n \in F$ satisfying that $\alpha_i = \beta_i^p$ for every $i \in \ldb 1,n \rdb$. In addition, as $G$ is $p$-divisible, we can take $h_1, \dots, h_n \in G$ such that $p h_i = g_i$ for every $i \in \ldb 1,n \rdb$. Now we can write $f = \big( \beta_n x^{h_n} + \dots + \beta_1 x^{h_1} \big)^p$ to infer that $f$ is not an irreducible element in $F[G]$. Hence $F[G]$ is an antimatter domain.
	\smallskip
	
	(2) Suppose now that $F$ is real-closed and also that $G$ is non-cyclic and has rank~$1$. As $G$ is a rank-$1$ torsion-free abelian group, by virtue of \cite[Section~24]{lF70} we can identify $G$ with an additive subgroup of $\qq$. As every additive subgroup of $\qq$ is the ascending union of infinite cyclic subgroups, the fact that $G$ is not cyclic guarantees the existence of a sequence $(g_n)_{n \ge 1}$ with terms in $G \cap \qq_{> 0}$ such that $G = \bigcup_{n \ge 1} \zz g_n$ and $\zz g_n \subsetneq \zz g_{n+1}$ for every $n \in \nn$. To argue that $F[G]$ is antimatter, fix a nonzero nonunit $f \in F[G]$ and write $f$ canonically as $f = \alpha_n x^{q_n} + \dots + \alpha_0 x^{q_0}$ for some nonzero $\alpha_0, \dots, \alpha_n \in F$ and $q_0, \dots, q_n \in G$ such that $q_n > \dots > q_0$. As $x^u$ is a unit in $F[G]$ for every $u \in G$, after replacing~$f$ by its associate $x^{-q_0} f$ if necessary, we can assume that $\text{ord} \, f = 0$. Fix now $g := g_N$ with $N \in \nn$ large enough so that $q_1, \dots, q_n \in \nn g$ and $q_n \ge 3g$. As a result, we can write $f = h(x^{g})$ for some $h \in F[x]$.  Since $F$ is a real closed field, every irreducible polynomial in $F[x]$ has degree at most~$2$. As $\deg \, h = q_n/g \ge 3$, we see that $h$ is not irreducible in $F[x]$. Therefore we can write $h(x) = b(x) c(x)$ for some polynomials $b(x), c(x) \in F[x] \setminus F$, and so $f = h(x^{g}) = b(x^{g}) c(x^{g})$. Since $\text{ord} \, h = \text{ord} \, f = 0$, neither $b(x)$ nor $c(x)$ is a monomial in $F[x]$ and, therefore, neither $b(x^{g})$ nor $c(x^{g})$ is a unit in $F[G]$. Thus, the equality $f = b(x^{g}) c(x^{g})$ ensures that $f$ is not an irreducible element in $F[G]$. Hence $F[G]$ is an antimatter domain.
	\smallskip 
	
	(3) For this part one can mimic the argument just given for part~(2), using the fact that every irreducible polynomial in $F[x]$ has degree $1$ when $F$ is algebraically closed.
\end{proof}

The interested reader can find in~\cite{ACHZ07} various methods to construct examples and classes of antimatter monoid algebras, including constructions that are parallel to those given in Proposition~\ref{prop:antimatter group algebras} .
\smallskip

It is worth noting that the group algebras characterized in Proposition~\ref{prop:group rings} as well as the group algebras given in Proposition~\ref{prop:antimatter group algebras} are either UFDs or antimatter domains, which are extreme conditions in the spectrum of atomicity. As of now, we have been unable to construct group algebras over fields that are atomic but not UFDs. We finish this subsection with the following question.

\begin{question}
	Do there exist a field $F$ and an abelian group $G$ such that the group algebra $F[G]$ is atomic but not a UFD?
\end{question}

\medskip
\subsection{A Look Inside Group Algebras That Are Not Hereditarily Atomic}
\label{subsec:inside group algebras that are not HA}

With notation as in Proposition~\ref{prop:group rings}, we know that if either $R$ is not an algebraic extension of $\ff_p$ for any $p \in \pp$ or $G$ is not an infinite cyclic group, then the group algebra $R[G]$ is not hereditarily atomic. In particular, the group algebras $\qq[\zz]$ and $\ff_p[\zz^2]$ are not hereditarily atomic, and it follows from parts~(1), (2), and~(3) of Proposition~\ref{prop:antimatter group algebras} that the group algebras $\ff_p[\qq]$, $\rr[\qq]$, and $\cc[\qq]$, respectively, are not even atomic. In this subsection, we find hereditarily atomic monoid algebras inside special cases of non-hereditarily atomic group algebras. We do this in Examples~\ref{ex:non-algebraic extension of F_p} and~\ref{ex:non-cyclic group} using, as a main tool, the next proposition. Then we say a few words about the known atomic monoid algebras that do not satisfy the ACCP, and we finish the subsection proposing potential candidates of atomic monoid algebras that do not satisfy the ACCP inside the non-hereditarily atomic group algebras $\ff_2[\qq]$ and $\ff_2[\zz^2]$.

\begin{prop} \cite[Proposition~2.1]{aG74} \label{prop:ACCP subrings}
	Let $R$ be an integral domain satisfying the ACCP. If a subring~$S$ of $R$ satisfies that $S^\times = R^\times \cap S$, then $S$ also satisfies the ACCP.
\end{prop}

When a field $F$ is not an algebraic extension of $\ff_p$ for any $p \in \pp$, 
it follows from Proposition~\ref{prop:group rings} that the group algebra $F[\zz]$ is not hereditarily atomic. However, even in this case, we can identify nontrivial subrings of $F[\zz]$ that are hereditarily atomic. 

\begin{ex} \label{ex:non-algebraic extension of F_p} 
	Let $F$ be a field that is not an algebraic extension of $\ff_p$ for any $p \in \pp$. Explicit non-atomic subrings of $F[\zz]$ were identified in Example~\ref{ex:group algebra ACCP but not HA}. On the other hand, for any $y \in F[\zz]$ that is transcendental over $F$, the subring $\zz[y]$ (resp., $\ff_p[y]$) of $F[\zz]$ is hereditarily atomic provided that $F$ has characteristic zero (resp., characteristic $p \in \pp$): indeed, in light of Proposition~\ref{prop:ACCP subrings}, every subring of $\zz[y]$ (resp., $\ff_p[y]$) satisfies the ACCP as it is a UFD and every subring $S$ of $\zz[y]$ (resp., $\ff_p[y]$) satisfies that $S^\times = \{\pm 1\} = \zz[y]^\times \cap S$ (resp., $S^\times = \ff_p^\times = \ff_p[y]^\times \cap S$).
\end{ex}

When $G$ is an Archimedean ordered group that is not cyclic, it follows from Proposition~\ref{prop:group rings} that the group algebra $\ff_p[G]$ is not hereditarily atomic for any $p \in \pp$. For each $p \in \pp$, in the next example we identify a class of hereditarily atomic subrings of $\ff_p[G]$ for every (non-cyclic) Archimedean ordered group $G$.

\begin{ex} \label{ex:non-cyclic group}
	Fix $p \in \pp$. Let $G$ be an Archimedean group and identify $G$ with an additive subgroup of $\rr$. Now fix $\epsilon \in \rr_{> 0}$, and consider the monoid $M := \{0\} \cup G_{\ge \epsilon}$. Take a sequence $(b_n)_{n \ge 1}$ with terms in $M$ such that  $(b_n + M)_{n \ge 1}$ is an ascending chain of principal ideals. Since $(b_n)_{n \ge 1}$ is a decreasing sequence of positive real numbers, it must converge. This implies that the sequence $(b_n - b_{n+1})_{n \ge 1}$, whose terms belong to~$M$, converges to zero, and so the fact that $\min M^\bullet = \epsilon$ guarantees the existence of $N \in \nn$ such that $b_n = b_N$ for every $n \ge N$. Hence the chain of principal ideals $(b_n + M)_{n \ge 1}$ stabilizes. Therefore~$M$ satisfies the ACCP. Because $M$ is a reduced monoid, it follows from \cite[Theorem~13]{AJ15} that the monoid algebra $\ff_p[M]$ also satisfies the ACCP. Using, once again, the fact that $M$ is reduced, we obtain that $\ff_p[M]^\times \cap S = \ff_p^\times \cap S = S^\times$ for every subring $S$ of $\ff_p[M]$. Hence $\ff_p[M]$ is hereditarily atomic by virtue of Proposition~\ref{prop:ACCP subrings}.
\end{ex}

Instead of hereditarily atomic monoid algebras, let us turn now to identify atomic monoid algebras that do not satisfy the ACCP (we keep on searching inside group algebras that are not hereditarily atomic). As mentioned in the introduction, not every atomic domain satisfies the ACCP. Let us briefly describe the first know example, which was constructed by Grams in~\cite{aG74}.

\begin{ex}[Grams' Construction] \label{ex:Grams' construction}
	Let $(p_n)_{n \ge 1}$ be the strictly increasing sequence whose underlying set is $\pp \setminus \{2\}$, and consider the monoid $M := \big\langle \frac 1{2^{n-1} p_n} \mid n \in \nn \big\rangle$ from part~(a) of Example~\ref{ex:atomic PM without ACCP}. Now fix a field $F$ and consider the localization $F[M]_S$ of the monoid algebra $F[M]$ at its multiplicative subset $S := \{f \in F[M]^* \mid \text{ord} \, f = 0 \}$. It follows from~\cite[Theorem~1.3]{aG74} that $F[M]_S$ is atomic. However, from the fact that $M$ does not satisfy the ACCP, one can readily deduce that $F[M]_S$ does not satisfy the ACCP. In~\cite[Theorem~3.3]{GL22}, Li and the author provided a generalization of the Grams' construction $F[M]_S$.
\end{ex}
Although Grams' construction does not lead to a monoid algebra (but to a localization of a monoid algebra), examples of atomic monoid algebras that do not satisfy the ACCP have also been found. Indeed, with notation as in Example~\ref{ex:Grams' construction}, it was proved by Li and the author in~\cite[Proposition~3.6]{GL22a} that when $F$ has characteristic zero, the ring of polynomials $R[x]$, where $R := F[M]_S$, is atomic but does not satisfy the ACCP (this seems to be the first known ring of polynomials with this property). The first atomic monoid algebra over a field that does not satisfy the ACCP was constructed by Zaks in~\cite[Example~2]{aZ82}, and it is a monoid algebra with an infinite-rank monoid of exponents. An alternative construction was recently given by Li and the author in~\cite[Section~4]{GL22}. An improved version of the latter construction was given even more recently by the same authors in~\cite[Theorem~5.15]{GL23}: it consists of a monoid algebra (over a field) whose monoid of exponents has finite rank (in fact, rank~$2$). We proceed to briefly outline both Zaks' example and the rank-$2$ example of atomic monoid algebras that do not satisfy the ACCP.

\begin{ex}[Zaks' Construction] \label{ex:Zaks' example}
	Let $F$ be a field, and consider the set of indeterminates $u,v,w$, and $x_n$ for all $n \in \nn$ over $F$. Now set $y_n := uv / (w^n x_n)$ for every $n \in \nn$, and let $R$ be the smallest subring of $F[u,v, w^{\pm 1}, x_n^{\pm 1}]_{n \in \nn}$ containing the set $\{u, v, w, x_n, y_n \mid n \in \nn \}$. It was proved by Zaks in~\cite[Example~2]{aZ82} that~$R$ is an atomic domain that does not satisfy ACCP. Moreover it was first observed in~\cite[Proposition~5.18]{GL23} that $R$ is indeed an infinite-rank monoid algebra. Indeed, after setting $M := \nn \oplus \nn \oplus \zz \oplus \zz^{\oplus \nn}$, we can think of the ring $F[u, v, w^{\pm 1}, x_n^{\pm 1}]$ as the monoid algebra $F[M]$ in a variable $z$, where $z^m = u^a v^b w^c (\prod_{n \in \nn} x_n^{d_n})$ for every element $m = (a,b,c, (d_n)_{n \in \nn}) \in M$ (here all but finitely many entries $d_n$ are zero), and so $R$ is the monoid algebra $F[N]$ in the variable $z$, where $N$ is the submonoid of $M$ generated by the elements $(1,0,0,0, \ldots)$, $(0,1,0,0, \ldots)$, $(0,0,1,0, \ldots)$, $e_{n+3} := (0,0,0,\dots, 0,1,0, \dots)$, and $f_n := (-n, 1, 1, 0, 0, \ldots, 0, -1, 0, \ldots)$, where the entry $1$ of $e_{n+3}$ and the entry $-1$ of $f_n$ appear in the $(n+3)$-th position for every $n \in \nn$. Observe that the monoid~$N$ has infinite rank, so $R$ is an infinite-rank monoid algebra.
\end{ex}

\begin{ex}[Gotti-Li's Construction] \label{ex:Gotti-Li's example}
	Let $(p_n)_{n \ge 1}$ be a strictly increasing sequence of odd primes such that $\sum_{n=1}^\infty \frac 1{p_n} < \frac{1}{3}$, and consider the additive submonoid $P = \big\langle \frac{1}{p_n} \mid n \in \nn \big\rangle$ of $\qq$. It is not hard to check that every $q \in P$ can be written in an essentially unique way as
	\begin{equation} \label{eq:unique decomposition}
		q = n_0 + \sum_{i=1}^\ell \frac{n_i}{p_i},
	\end{equation}
	where $n_0, n_1, \dots, n_\ell \in \nn_0$ and $n_i \in \ldb 0, p_i - 1 \rdb$ for every $i \in \ldb 1, \ell \rdb$. From the uniqueness of such a representation, we can deduce that $\mathcal{A}(P) = \big\{ \frac 1{p_n} \mid n \in \nn \big\}$ and also that $P$ satisfies the ACCP (see \cite[Example~2.1]{AAZ90} and \cite[Proposition~4.2]{fG22}). Now consider the following subset of $P$:
	\[
		A := \bigg\{ \frac 1{p_{j_k}} + \sum_{i=1}^\ell \frac{1}{p_{j_i}}  \ \Big{|} \ k, \ell \in \nn \text{ with } k \in \ldb 1, \ell \rdb \text{ and } j_1 < j_2 < \cdots < j_\ell \bigg\};
	\]
	that is, each element of $A$ is the sum of finitely many atoms of $P$, where exactly one of them repeats, and it repeats exactly twice. In addition, let $\beta$ be an irrational number such that $\beta > 1$, and consider the following subset of $\nn_0 \beta + \qq$:
	\[
		B := \{ \beta \} \bigcup \bigg\{ \beta - \sum_{i=1}^\ell \frac{1}{p_i}  \ \Big{|} \ \ell \in \nn \bigg\}.
	\]
	Now let $M$ denote the positive monoid generated by the set $A \cup B$, and observe that $M$ has rank~$2$. It follows from \cite[Theorem~5.15]{GL23} that, for any field $F$, the monoid algebra $F[M]$ is atomic but does not satisfy the ACCP.
\end{ex}

Let $M$ be one of the monoids described in Example~\ref{ex:atomic PM without ACCP} or one of the monoids constructed in Proposition~\ref{prop:atomic monoid that is not ACCP}. We have seen in the same example/proposition that~$M$ is an atomic monoid that does not satisfy the ACCP. Therefore the monoid algebra $F[M]$ does not satisfy the ACCP for any field~$F$. In connection with Examples~\ref{ex:Zaks' example} and~\ref{ex:Gotti-Li's example}, we pose the following questions.

\begin{question} \label{quest:new atomic monoid algebras without ACCP} \hfill
	\begin{enumerate}
		\item For which of the monoids $M$ described in Example~\ref{ex:atomic PM without ACCP}, is the monoid algebra $\ff_2[M]$ atomic?
		\smallskip
		
		\item Let $M_2$ be the rank-$2$ monoid constructed in the proof of Proposition~\ref{prop:atomic monoid that is not ACCP}. Is the monoid algebra $\ff_2[M_2]$ atomic?
	\end{enumerate}
\end{question}

We have chosen $\ff_2$ in the statement of Question~\ref{quest:new atomic monoid algebras without ACCP} only for the sake of simplicity. The corresponding questions for any other fields are still open, and answers for other fields are equally interesting in the following way. A positive answer to part~(2) of Question~\ref{quest:new atomic monoid algebras without ACCP} would yield the second example of a two-dimensional atomic monoid algebra (over a field) not satisfying the ACCP (the first one is the one in Example~\ref{ex:Gotti-Li's example}). More importantly, a positive answer to part~(1) of Question~\ref{quest:new atomic monoid algebras without ACCP} (for any of the three monoids) would yield the first ever known one-dimensional atomic monoid algebra not satisfying the ACCP. Finally, any negative answer to any of the parts of Question~\ref{quest:new atomic monoid algebras without ACCP} would identify an atomic monoid with a monoid algebra over a field that is not atomic (the first and only examples known were constructed by Coykendall and the author in~\cite{CG19}, partially answering a question about the atomicity of monoid algebras posed by Gilmer~\cite[page 189]{rG84} back in the eighties.

\medskip
\subsection{Hereditary ACCP and the Bounded Factorization Property}
\label{subsec:HACCP and the BFP}

In the same way that we have defined hereditarily atomic domains and hereditary ACCP monoids, we can go further and define hereditary ACCP domains. Following~\cite{CGH21}, we say that an integral domain~$R$ is \emph{hereditary ACCP} provided that every subring of~$R$ satisfies the ACCP. By definition, every hereditary ACCP domain must satisfy the ACCP. Another relevant class of integral domains satisfying the ACCP is that consisting of all bounded factorization domains. For an atomic domain $R$ and a nonzero nonunit $r \in R$, a \emph{factorization length} of $r$ is the number of irreducibles (counting repetitions) in a factorization of $r$, and we let $\mathsf{L}(r) \subseteq \nn$ denote the set of all factorization lengths of $r$. Following~\cite{AAZ90}, we say that an atomic domain $R$ is a \emph{bounded factorization domain} (BFD) if $\mathsf{L}(r)$ is finite for every nonzero nonunit $r \in R$. It is clear that every UFD is a BFD, and it is not hard to show that every BFD satisfies the ACCP (see \cite[Corollary~1]{fHK92}). As a result, we obtain the diagram of nested classes of atomic domains illustrated in Figure~\ref{fig:AAZ chain}, which was first considered in the landmark paper~\cite{AAZ90} (as part of a larger diagram) and since then has established a methodology to study the deviation of atomic domains (and also monoids) from satisfying the unique factorization property.

\begin{figure}[h]
	\begin{equation} \label{diag:AAZ's atomic chain until bi-FFM}
		\begin{tikzcd}
			\textbf{ UFD } \arrow[r, Rightarrow]  \arrow[red, r, Leftarrow, "/"{anchor=center,sloped}, shift left=1.7ex] & \textbf{ BFD } \arrow[r, Rightarrow]  \arrow[red, r, Leftarrow, "/"{anchor=center,sloped}, shift left=1.7ex] & \textbf{ ACCP domain}  \arrow[r, Rightarrow] \arrow[red, r, Leftarrow, "/"{anchor=center,sloped}, shift left=1.7ex]  & \textbf{ atomic domain}
		\end{tikzcd}
	\end{equation}
	\caption{The implications in the diagram determine nested classes of atomic domains. The broken red double arrows indicate that none of the displayed implications is reversible (counterexamples witnessing the broken arrows, from left to right, are given in~\cite[Example~4.7]{AG22}, Example~\ref{ex:HACCP and BF stronger than ACCP}(b) below, and~\cite[Theorem~1.3]{aG74}.}
	\label{fig:AAZ chain}
\end{figure}
\noindent The property of being hereditary ACCP and that of being a BFD are strictly stronger than the property of satisfying the ACCP. Indeed, none of the two former properties implies the other in the context of integral domains (or even in the context of monoid algebras over fields). The following examples shed some light upon this observation.

\begin{ex} \label{ex:HACCP and BF stronger than ACCP} \hfill
	\begin{enumerate}
		\item[(a)] The ring of Laurent polynomials $\qq[x^{\pm 1}]$ is a UFD and so a BFD. However, as $\qq[x^{\pm 1}]$ has characteristic zero, it follows from Proposition~\ref{prop:group rings} that it is not hereditarily atomic. Hence $\qq[x^{\pm 1}]$ is not hereditary ACCP.
		\smallskip
		
		\item[(b)] Consider the additive submonoid $M := \big\langle \frac 1p \mid p \in \pp \big\rangle$ of $\qq$. It is well known and not hard to prove that $M$ is atomic with $\mathcal{A}(M) = \big\{ \frac 1p \mid p \in \pp \big\}$. Furthermore, $M$ satisfies the ACCP (see \cite[Example~2.1]{AAZ90} or the proof of \cite[Proposition~4.2]{fG22} for more details). Now fix $p \in \pp$. As $M$ is a reduced monoid satisfying the ACCP, it follows from \cite[Theorem~13]{AJ15} that the monoid algebra $\ff_p[M]$ also satisfies the ACCP. Moreover, since every subring $S$ of $\ff_p[M]$ satisfies that $S^\times = \ff_p^\times = \ff_p[M]^\times \cap S$, it follows from Proposition~\ref{prop:ACCP subrings} that $\ff_p[M]$ is hereditary ACCP. However, $\ff_p[M]$ is not a BFD: indeed, $x = \big( x^{1/p}\big)^p$ for every prime $p$, and so $\mathsf{L}(x) = \pp$.
	\end{enumerate}
\end{ex}

\bigskip
\section*{Acknowledgments}

While working on this paper, the author was kindly supported by the NSF awards DMS-1903069 and DMS-2213323.
\bigskip

\bigskip

\end{document}